\documentclass[12pt,leqno]{amsart}
\usepackage[colorlinks,citecolor=red,pagebackref,hypertexnames=true]{hyperref}
\usepackage[backrefs,msc-links,nobysame]{amsrefs}
\usepackage{amsmath,amstext,amsthm,amssymb,amsxtra}

\usepackage{txfonts} 
\usepackage[T1]{fontenc}
\usepackage{lmodern}

\usepackage{eulervm}    

 \usepackage{mathtools}
\mathtoolsset{showonlyrefs,showmanualtags}

\usepackage{amsrefs}

\allowdisplaybreaks

\theoremstyle{plain} 
\newtheorem{lemma}[equation]{Lemma}

\newtheorem*{fs}{Theorem}

\theoremstyle{definition}

\theoremstyle{remark}
\newtheorem{remark}[equation]{Remark}

\numberwithin{equation}{section}

\def\norm#1.#2.{\lVert#1\rVert_{#2}}
\def\Norm#1.#2.{\bigl\lVert#1\bigr\rVert_{#2}}
\def\NOrm#1.#2.{\Bigl\lVert#1\Bigr\rVert_{#2}}
\def\NORm#1.#2.{\biggl\lVert#1\biggr\rVert_{#2}}
\def\NORM#1.#2.{\Biggl\lVert#1\Biggr\rVert_{#2}}

\def\ip#1,#2,{\langle #1,#2\rangle}
\def\Ip#1,#2,{\bigl\langle#1,#2\bigr\rangle}
\def\IP#1,#2,{\Bigl\langle#1,#2\Bigr\rangle}

\def\R{\mathbb R}

\def\abs#1{\lvert#1\rvert}
\def\Abs#1{\bigl\lvert#1\bigr\rvert}

\def\ind{\mathbf 1}

\DeclareMathOperator*\essinf{ess\,inf}
\begin{document}
\title {The endpoint Fefferman-Stein inequality for the strong maximal function}
\begin{abstract} Let $M_nf$ denote the strong maximal function of $f$ on $\R^n$, that is the maximal average of $f$ with respect to $n$-dimensional rectangles with sides parallel to the coordinate axes. For any dimension $n\geq 2$ we prove the natural endpoint Fefferman-Stein inequality for $M_n$ and any strong Muckenhoupt weight $w$:
$$w(\{x\in\R^n: M_nf(x)>\lambda\})\lesssim_{w,n}  \int_{\R^n} \frac{|f(x)|}{\lambda}\Big(1+\big(\log^+ \frac{|f(x)|}{\lambda} \big)^{n-1}\Big) M_nw(x) dx.$$
This extends the corresponding two-dimensional result of T. Mitsis.
\end{abstract}


\author[T. Luque]{Teresa Luque}
\address{Teresa Luque, Departamento de Analisis Matem\'atico, Facultad de Matem\'aticas, Universidad de Sevilla, 41080 Sevilla, Spain}
\email{tluquem@us.es}

\author[I. Parissis]{Ioannis Parissis}
\address{Ioannis Parissis, Department of Mathematics, Aalto University, P.O.Box 11100, FI-00076 Aalto, Finland}
\email{ioannis.parissis@gmail.com}
\thanks{T.L. is supported by the Spanish Ministry of Economy and Competitiveness grant BES-2010-030264}
\thanks{I.P. is supported by the Academy of Finland, grant 138738.}

\maketitle

\section{Introduction}
\subsection*{The strong maximal function} Let $\mathfrak R_n$ denote the family of all rectangles in $\mathbb R^n$ with sides parallel to the coordinate axes. For a locally integrable function $f$ on $\R^n$ we will denote by $M_n f$ the \emph{strong maximal function}:
\begin{align*}
 M_nf(x)\coloneqq \sup_{\substack{R\in\mathfrak R_n\\ R\ni x}} \frac{1}{|R|} \int_R |f(y)|dy,\quad x\in\mathbb R^n.
\end{align*}
Here $|S|$ denotes the $n$-dimensional Lebesgue measure of a set $S\subset \R^n$. We will sometimes use the same notation for the $(n-1)$-dimensional Lebesgue measure, but this will be clear from the context.

The endpoint behavior of $M_n$ close to $L^1$ is given by the classical theorem of Marcinkiewicz, Jessen and Zygmund, \cite{JMZ}:
\begin{align}\label{e.fsnw}
 |\{x\in\R^n: M_nf(x)>\lambda\}|\lesssim_n \int_{\R^n} \frac{|f(x)|}{\lambda} \Big(1+\big(\log^+\frac{|f(x)|}{\lambda}\big)^{n-1}\Big)dx,
\end{align}
where $\log^+ t\coloneqq  \max(0,\log t)$. Inequality \eqref{e.fsnw} implies the strong differentiation of the integral of all functions $f\in L(1+(\log^+L)^{n-1})(\R^n)$, that is, all functions $f$ on $\R^n$ such that:
\begin{align*}
 \int_{\R^n}|f(x)| \Big(1+\big(\log^+|f(x)|\big)^{n-1}\Big)dx<+\infty.
\end{align*}
The strong maximal theorem cannot be improved, that is, the function $t(1+(\log^+ t)^{n-1})$ on the right hand side of \eqref{e.fsnw} cannot be replaced by any slower increasing function. Remember that the usual Hardy-Littlewood maximal function is the maximal average of $f$ with respect to all $n$-dimensional Euclidean cubes, or balls and it maps $L^1(\R^n)$ to $L^{1,\infty}(\R^n)$. The important difference to be noted here is that the strong maximal function is an $n$-parameter maximal average, in contrast to the the usual one-parameter Hardy-Littlewood maximal function, and this difference is reflected in the strong maximal theorem \eqref{e.fsnw} which requires extra logarithmic scales of integrability. The original proof from \cite{JMZ} relies on the observation that $M_n$ can be viewed as a composition of $n$ one-dimensional Hardy-Littlewood maximal operators. One then appeals to the one-dimensional theory to get \eqref{e.fsnw} together with its strong $L^p(\R^n)$ counterparts. A more geometric point of view was introduced by the work of C\'ordoba and R. Fefferman, \cite{CF}, who gave a proof of \eqref{e.fsnw} by means of a geometric covering argument. This is in a sense a dual point of view where the $n$-parameter composition of operators is replaced by induction on the dimension. The importance of the C\'ordoba-Fefferman geometric proof of the strong maximal theorem is highlighted by the fact that the usual Besicovitch covering argument fails when applied to families of rectangles having arbitrary eccentricities.

\subsection*{Strong weights}A \emph{weight} $w$ will be a locally integrable, non-negative function on $\mathbb R^n$. We will say that $w$ belongs to the class $A_p ^*$, $1<p<\infty$, whenever
\begin{align*}
 [w]_{A_p ^*}\coloneqq \sup_{R\in\mathfrak R_n} \bigg(\frac{1}{|R|}\int_R w \bigg)  \bigg( \frac{1}{|R|}\int_R w^{1-p'} \bigg)^{p-1}<+\infty.
\end{align*}
In this case we will say that $w$ is a \emph{strong $A_p$-weight}. For $p=1$ the class $A_1 ^*$ is defined by the condition
\begin{align*}
\frac{1}{|R|} \int_R w \leq C \essinf_{x\in R} w(x),\quad \mbox{for almost every} \ x\in R, \quad R\in\mathfrak R_n,
\end{align*}
which is equivalent to saying that $M_n w\leq C w$ almost everywhere in $\R^n$. The smallest constant $C>0$ in the previous inequality is the $A_1 ^*$-constant of the weight, denoted by $[w]_{A_1 ^*}$.

It follows by H\"older's inequality that the $A_p ^*$ classes are increasing, that is, for $1\leq p\leq q <\infty$ we have $A_p ^* \subset A_q ^*$. We define the class $A_\infty ^*$ by means of
\begin{align*}
 A_\infty ^* \coloneqq \bigcup_{p> 1} A_p ^*.
\end{align*}
It is equivalent to define the class $A_\infty ^*$ by the following property: there exist constants $\delta,c>0$ such that, given any rectangle $R\in\mathfrak R_n$ and a measurable subset $S\subset R$, then
\begin{align}\label{e.doubling}
 \frac{w(S)}{w(R)} \leq c \bigg(\frac{|S|}{|R|}\bigg)^\delta.
\end{align}
An important feature of strong $A_\infty$-weights is that if we fix any $t\in\R$ then the weight
\begin{align*}
w^t(x')\coloneqq w(x',t),\quad x'\in\R^{n-1},
\end{align*}
is an $A_\infty ^*$-weight on $\R^{n-1}$, \emph{uniformly} in $t\in\R$. In practice, uniformly means that all the constants connected with the properties of the $A_\infty ^*$-weight $w^t$ can be taken to be independent of $t$. For these and other properties of strong Muckenhoupt weights see for example \cite{BaKu} or \cite{GaRu}*{Chapter IV}.
\begin{remark}\label{r.esmall} Let $w\in A_\infty ^*$. By the previous discussion we see that there exists some $\epsilon =\epsilon(w)>0$ such that, for every rectangle $R\in\mathfrak R_n$ and all measurable sets $F\subset \R^n$, we have
\begin{align*}
 |R\cap F|\leq \epsilon|R| \Rightarrow w(R\cap F) \leq \frac{1}{2} w(R) \Rightarrow w(R\setminus F)\geq \frac{1}{2} w(R).
\end{align*}
In fact, it suffices to choose $\epsilon>0$ so that $c\epsilon^\delta\leq \frac{1}{2}$, where $c,\delta$ are the constants associated to $w\in A_\infty ^*$ from \eqref{e.doubling}. Since for any $t\in \R$ the weight $w^t\coloneqq w(\cdot,t)$ is an $A_\infty ^*$-weight on $\R^{n-1}$, uniformly in $t$, the $\epsilon>0$ can be chosen sufficiently small so that we have the previous property also for $w^t$, rectangles $R'\in\mathfrak R_{n-1}$ and sets $F'\subset \R^{n-1}$, uniformly in $t$. We will use this remark several times in what follows.
\end{remark}

It is known that $M_n$ is bounded on $L^p(w)$, $1<p<\infty$ if and only if $w\in A_p ^*$. This result follows again by an appeal to the one dimensional theory. For the necessity of the $A_p ^*$ condition, one argues as in the case of the usual $A_p$ weights. The corresponding endpoint bound is also true: the operator $M_n$ satisfies the distributional estimate
\begin{align*}
 w(\{x\in\R^n: M_nf(x)>\lambda\})\lesssim_{w,n}  \int_{\R^n} \frac{|f(x)|}{\lambda} \Big(1+\big(\log^+\frac{|f(x)|}{\lambda}\big)^{n-1}\Big)w(x)dx
\end{align*}
whenever $w\in A_1 ^*$. For these results see \cite{BaKu}*{Theorems 2.1 and 2.3}.

\subsection*{Weighted strong maximal function} For $w\in A_\infty ^*$ we will also consider the weighted strong maximal function $M^w _n$, defined with respect to $w$:
\begin{align*}
 M_n ^wf(x) \coloneqq  \sup_{\substack{R\in\mathfrak R_n\\ R\ni x}} \frac{1}{w(R)}\int_R |f(y)| w(y) dy.
\end{align*}
For the weighted strong maximal function $M_n ^w$, R. Fefferman showed in \cite{F} that it maps $L^p(w)$ to $L^p(w)$ whenever $w\in A_\infty ^*$:
\begin{align}\label{e.weighted}
 \norm M_n ^w f.L^p(w). \leq_{w,n} c_{p,n} \norm f.L^p(w). , \quad 1<p\leq \infty.
\end{align}
Furthermore we have the asymptotic estimate
\begin{align}\label{e.strong}
c_{p,n}=O_{n}( (p-1)^{-n})\quad\mbox{as}\quad p\to 1^+.
\end{align}

The behavior of the constants is not explicitly studied in \cite{F} but follows by a close examination of the proof and the standard norm bounds from Marcinkiewicz interpolation. See also \cite{LOSH}*{Lemma 4}. The endpoint bound for $M^w _n$ is also true, namely $M^w _n$ satisfies
\begin{align}\label{e.a1}
  w(\{x\in\R^n: M^w _nf(x)>\lambda\})\lesssim_{w,n}  \int_{\R^n} \frac{|f(x)|}{\lambda} \Big(1+\big(\log^+\frac{|f(x)|}{\lambda}\big)^{n-1}\Big)w(x)dx
\end{align}
whenever $w\in A_\infty ^*$. This was proved by Jawerth and Torchinsky, \cite{JT}, and independently by Long and Shen \cite{LOSH}.

\subsection*{Fefferman-Stein inequality} By this we mean in general an inequality of the form
\begin{align*}
 \int_{\R^n} (\mathcal Mf)^p w \lesssim_{w,n} \int_{\R^n} |f|^p \mathcal M w dx,\quad 1<p<\infty,
\end{align*}
where $\mathcal M$ denotes some maximal operator. Inequalities of this type are important since, among other things, they can be used to derive the boundedness of vector-valued maximal operators. In fact this inequality was first proved for the Hardy-Littlewood maximal function by C. Fefferman and Stein, \cite{FS}, for \emph{every} non-negative, locally integrable weight $w$. The main application in \cite{FS} was exactly the vector-valued extension of the classical Hardy-Littlewood maximal theorem. For the strong maximal function the same inequality is true provided that $w\in A_\infty ^*$; see \cite{Lin} for a direct proof of this result and  also \cite{Per}, where the Fefferman-Stein inequality is obtained as a corollary of a more general two weight-norm inequality. Observe that, as in the case of \eqref{e.weighted}, we need some extra assumption on the weight in order to prove the Fefferman-Stein inequality for the strong maximal function. This should be contrasted to the corresponding result for the Hardy-Littlewood weighted maximal function, as well as to the Fefferman-Stein inequality for the Hardy-Littlewood maximal function, where no assumption on the weight is needed.

The form of the endpoint Fefferman-Stein inequality  depends on the corresponding unweighted endpoint properties of the maximal operator under study. For the usual Hardy-Littlewood maximal function $M_Q$ the right statement is
\begin{align*}
 w(\{x\in\R^n: M_Qf(x)>\lambda\})\lesssim_{n} \frac{1}{\lambda} \int_{\R^n} |f(x)| M_Q w(x) dx.
\end{align*}
The natural endpoint Fefferman-Stein inequality for the strong maximal function was proved by Mitsis, \cite{Mitsis}, in dimension $n=2$. In particular Mitsis showed that
\begin{align*}
 w(\{x\in\R^2: M_nf(x)>\lambda\})\lesssim_{w}  \int_{\R^2} \frac{|f(x)|}{\lambda}\Big(1+ \log^+ \frac{|f(x)|}{\lambda}  \Big) M_n w(x) dx.
\end{align*}

The main result of the current paper is the extension of the endpoint Fefferman-Stein inequality for the strong maximal function to all dimensions:
\begin{fs} Let $w\in A_\infty ^*$. For all dimensions $n\geq 1$ we have
\begin{align*}
 w(\{x\in\R^n: M_nf(x)>\lambda\})\lesssim_{n,w}  \int_{\R^n} \frac{|f(x)|}{\lambda}\Big(1+\big(\log^+ \frac{|f(x)|}{\lambda} \big)^{n-1}\Big) M_n w(x) dx.
\end{align*}
\end{fs}
By interpolation, the Fefferman-Stein inequality of our main theorem  above implies the strong $L^p$-version of the Fefferman-Stein inequality from \cite{Lin}, \cite{Per}. Furthermore, since every $A_1 ^*$-weight is an $A_\infty ^*$-weight, we recover the endpoint inequality \eqref{e.a1} for $A_1 ^*$-weights. 

It should be noted that the proof of Mitsis in \cite{Mitsis} uses the combinatorics of two-dimensional rectangles, which allow one to get favorable estimates for the measures
\begin{align*}
 |\{x\in R_k: \sum_{k=1} ^N \ind_{R_k}(x)=\ell\}|;
\end{align*}
here $\{R_k\}_{1\leq k \leq N}$ is a sequence of rectangles which satisfy a certain sparseness property and $\ell$ is any integer in $\{1,2,\ldots,N\}$. These combinatorics do not seem to be readily available in higher dimensions and so we adopt a different approach, which relies on the boundedness of the weighted strong maximal function $M_n ^w$ and the precise estimate for its norm, \eqref{e.strong}. In particular, our approach is inspired by the arguments in \cite{LOSH}, a paper which seems to have been overlooked by most of the works on the weighted inequalities for the strong maximal function.
\subsection*{Acknowledgments} This work was done while T.L was visiting T.Hyt\"onen at University of Helsinki. The authors would like to thank him for his generosity and hospitality. We also want to thank C. P\'erez and T. Mitsis for some valuable discussions on the subject of this paper.

\section{Notation} We write $A\lesssim B$ if $A\leq C B$ for some numerical constant $C>0$. In order to indicate the dependence of the constant on some parameter $n$ (say), we write $A\lesssim_n B$. Similarly, $A\simeq B$ means that $A\lesssim B$ and $B\lesssim A$.

\section{Some geometry of $n$-dimensional rectangles} In this section we recall some sparseness properties of $n$-dimensional rectangles, introduced in \cite{CF}. Here we adopt the slightly different approach from \cite{LOSH}. In fact, both Lemmas in this section are mentioned in \cite{LOSH}. However, we present the proofs for the sake of completeness.

For $t\in \R$ and $E\subset \R^n$ we introduce the slice operator
\begin{align*}
P_t(E)\coloneqq \{x' \in \R^{n-1}: (x',t)\in E\}.
\end{align*}
Thus $P_t$ is the `slice' of $E$ by a hyperplane perpendicular to the $n$-th coordinate axis at level $t\in \R$. The $(n-1)$-dimensional projection is
\begin{align*}
P_\parallel(E)\coloneqq\{x'\in \R^{n-1}: (x',t)\in E\quad\mbox{for some}\quad t\in\R\}.	
\end{align*}
We will also use the one-dimensional projection $P^\perp$ defined for $E\subset \R^n$ as
\begin{align*}
P^\perp(E)\coloneqq \{t\in\R: (x',t)\in E\quad\mbox{for some}\quad x'\in\R^{n-1} \}.
\end{align*}
If $R\in\mathfrak R_n$ observe that we have
\begin{align*}
	R=P_\parallel(R)\times P^\perp(R)=P_t(R)\times P^\perp(R),\quad\mbox{for all}\quad t\in P^\perp(R).
\end{align*}
For any interval $I\subset \R$, let $I^*$ be the interval with the same center and three times the length of $I$, $|I^*|=3|I|$. For $R\in\mathfrak R_n$ we then use the notation
\begin{align*}
R^*\coloneqq P_\parallel(R)\times (P^\perp(R))^*.
\end{align*}
Thus $R^*$ is the rectangle with the same center as $R$ and whose sides parallel to the first $n-1$ coordinate axes have the same lengths as the corresponding sides of $R$; the side of $R$ which is parallel to the $n$-th coordinate axis has length equal to three times the length of the corresponding side of $R$.

Let $\mathcal R = \{R_k\}_{1\leq k\leq N}$ be a finite sequence of rectangles from $\mathfrak R_n$. We will say that $\mathcal R$ satisfies the sparseness property \eqref{e.p2} if
\begin{align*}\tag{$P_2$}\label{e.p2}
 \begin{cases} P^\perp (R_1)\geq P^\perp(R_2)\geq \cdots\geq P^\perp(R_N),
\\
|R_k \cap \bigcup_{j< k} R_j ^* |\leq \epsilon |R_k|, \quad k=1,2,\ldots,N. \end{cases}
\end{align*}
Here $0<\epsilon<1$ will be assumed sufficiently small in various parts of the arguments below.

For $t\in\R$ we now consider the collection $\mathcal T (t)=\mathcal T= \{P_t(R_k)\}_{1\leq k \leq N} \subset\mathfrak R_{n-1}$ which is produced by slicing all the $n$-dimensional rectangles of $\mathcal R $ by a hyperplane perpendicular to the $n$-th coordinate axis, at the level $t$. The collection $\mathcal T$ depends on $t$ but we will many times suppress this fact in what follows. The main point about the collections $\mathcal R$ and $\mathcal T$ is contained in the following standard fact.
\begin{lemma}\label{l.slice} Suppose that the sequence $\mathcal R= \{R_k\}_{1\leq k\leq N}$ has the sparseness property \eqref{e.p2}. Then, for all $t\in\R$, the $(n-1)$-dimensional collection of rectangles $\mathcal T(t)=\{P_t(R_k)\}_{1\leq k \leq N}$ has the sparseness property \eqref{e.p1}, uniformly in $t$:
\begin{align*}\tag{$P_1$}\label{e.p1}
 |P_t(R_k) \cap \bigcup_{j<k} P_t(R_j) |\leq \epsilon |P_t(R_k)|,  \quad k=1,2,\ldots,N.
\end{align*}
\end{lemma}
\begin{proof} We fix some $1\leq k \leq N$ and $t\in P^\perp(R_k)$. Denoting $J\coloneqq \{j<k: P_t(R_k)\cap P_t(R_j)\neq \emptyset\}$ we have by the second condition in \eqref{e.p2} that
\begin{align}\label{e.first}
 \epsilon |R_k| &\geq |R_k \cap \bigcup_{j<k} R_j ^*|=|\bigcup_{j<k} (R_k \cap R_j ^*)|\geq|\bigcup_{j\in J} (R_k \cap R_j ^*)|
\end{align}
Observe that for $j\in J$ we have that $\emptyset \neq P^\perp(R_k)\cap P^\perp(R_j)\ni t$ and by the first condition in \eqref{e.p2} we have $|P^\perp(R_j)|\geq |P^\perp(R_k)|$. A moment's reflection shows that if $I_1,I_2$ are intervals in $\R$, $|I_2|\geq |I_1|$ and $I_1\cap I_2\neq \emptyset$ then $I_1\subseteq I_2^*$. We conclude that $P^\perp(R_k)\subseteq P^\perp(R_j ^*)$. Thus the $n$-dimensional rectangle $R_k\cap R_j^*$ is of the form $P^\perp(R_k)\times P_\parallel(R_k \cap R_j ^*)$. However, $j\in J$ implies that $P_t(R_k\cap R_j)=P_t(R_k)\cap P_t(R_j)\neq \emptyset$, so we conclude that $P_\parallel(R_k\cap R_j ^*)=P_t(R_k\cap R_j)$ and thus
\begin{align}\label{e.second}
 R_k\cap R_j ^*= P^\perp(R_k) \times P_t(R_k\cap R_j).
\end{align}
Now estimate \eqref{e.first} and identity \eqref{e.second} give
\begin{align*}
\epsilon |P^\perp(R_k)|\times |P_t(R_k)|= \epsilon |R_k|&\geq \Abs{\bigcup_{j\in J} P^\perp(R_k)\times P_t(R_k\cap R_j)}
\\
& = |P^\perp(R_k)|\times \Abs{\bigcup_{j\in J} P_t(R_k\cap R_j)}
\\
&=  |P^\perp(R_k)|\times \Abs{ P_t(R_k)\cap\bigcup_{j\in J} P_t(R_j)}
\\
&=  |P^\perp(R_k)|\times | P_t(R_k)\cap\bigcup_{j<k} P_t(R_j)|.
\end{align*}
This proves the lemma for $t\in P^\perp(R_k)$ while for $t\notin P^\perp(R_k)$ the conclusion follows trivially.\end{proof}
The next lemma gives a precise quantitative bound on the overlap of the rectangles in $\mathcal R$ under the sparseness property \eqref{e.p2}.
\begin{lemma}\label{l.t*1} Let $w\in A_\infty ^*$ and suppose that the finite sequence $\mathcal R=\{R_k\}_{1\leq k \leq N}\subset \mathfrak R_n$ satisfies property \eqref{e.p2} with $\epsilon$ sufficiently small, depending on the weight $w$. We set $\Omega\coloneqq \cup_{k=1} ^N R_k$. For $1<  p < \infty$ we have
 \begin{align*}
  \bigg( \int_\Omega \Abs{\sum_{k=1} ^N \ind_{R_k}}^pw(x)dx \bigg)^\frac{1}{p} \lesssim_{w,n} c_{p,n} w(\Omega)^\frac{1}{p}
 \end{align*}
with $c_{p,n}= O_n(p^{n-1})$ as $p\to +\infty$.
\end{lemma}

\begin{proof} For a sequence $\{R_k\}_{1\leq k \leq N}$ as before, consider the sequence $\mathcal T(t)$ of $(n-1)$-dimensional rectangles, by slicing the collection $\mathcal R$ with a hyperplane perpendicular to the $n$-th coordinate axis, at level $t\in \R$. Let $\Omega_t\coloneqq P_t(\Omega)$ denote the corresponding slice of $\Omega$ at level $t$ and set $T_k\coloneqq P_t(R_k)$ in order to simplify the notation. By Lemma \ref{l.slice} the collection $\mathcal T(t)=\{T_k\}_{1\leq k \leq N}$ has the property \eqref{e.p1}. We set $E_k\coloneqq T_k\setminus \cup_{j<k} T_j$. For fixed $t\in \R$, the function $w^{t}(x')=w(x',t),\ x'\in \R^{n-1}$, is an $A_\infty ^*$-weight in $\mathbb R^{n-1}$, uniformly in $t\in\R$; see \cite{F}. By the property \eqref{e.p1} and the fact that $w^t\in A_\infty ^*$ uniformly in $t$, we will have that $w^t (T_k)\geq w^t (E_k)\geq \frac{1}{2} w^t (T_k)$ if $\epsilon>0$ was selected sufficiently small in property \eqref{e.p2}, and thus also in \eqref{e.p1}, according to Remark \ref{r.esmall}.

Define the linear operator
\begin{align*}
 L_{w^t} f(x')\coloneqq \sum_{k=1} ^N \frac{1}{w^t(T_k)} \bigg(\int_{T_k} f(y')w^t (y') dy' \bigg) \ind_{E_k}(x'), \quad x' \in \R^{n-1}.
\end{align*}
For any locally integrable function $f$ on $\R^{n-1}$ we have that $L_{w^t}f(x') \leq M_{w^t}f(x'),\ x'\in\R^{n-1}$. Also observe that for $f,g$ locally integrable we have
\begin{align*}
 \int_{\Omega_t} L_{w^t} f(x') g(x')w^t (x')dx' &=\int_{\Omega_t} \sum_{k=1} ^N \frac{1}{w^t(T_k)} \bigg( \int_{E_k} g(y')w^t (y')dy' \bigg)\ind_{R_k}(x') f(x')w^t (x') dx'
\\
&\eqqcolon \int_{\Omega_t} L_{w^t} ^*  g(x') f(x') w^t(x')dx'.
\end{align*}
Furthermore $L_{w^t} ^*(\ind_{\Omega_t} ) = \sum_{k=1} ^N \frac{w^t(E_k)}{w^t(T_k)}\ind_{T_k}\geq \frac{1}{2} \sum_{k=1} ^N \ind_{T_k}$. For any locally integrable function $g$ on $\R^{n-1}$ we thus have
\begin{align*}
 \int_{\Omega_t} g(x') \sum_{k=1} ^N \ind_{T_k}(x') w^t (x') dx'& \lesssim  \int_{\Omega_t} g(x') L_{w^t} ^* (\ind_{\Omega_t})(x') w^t(x') dx'
\\
&=\int_{\Omega_t} L_{w^t} g(x') w^t(x')dx'\leq \norm M_{w^t} g.L^{p'}(w^t,\R^{n-1}). w^t(\Omega_t)^\frac{1}{p}
\\
&\lesssim_{w,n} (p'-1)^{-(n-1)} \norm g.L^{p'}(w^t,\R^{n-1}). w^t(\Omega_t)^\frac{1}{p},
\end{align*}
by \eqref{e.strong}. Taking the supremum over $g\in L^{p'}(\R^{n-1})$ with $\norm g.L^{p'}(w^t,\R^{n-1}).\leq 1$ gives the estimate
\begin{align}
\int_{\Omega_t} \Abs{\sum_{k=1} ^N \ind_{T_k}(x')}^p w^t (x')dx \lesssim_{w,n}  p^{(n-1)p} w^t(\Omega_t),
\end{align}
as $p\to \infty$. It is essential to note here that this estimate is uniform in $t\in \R$. Thus integrating over $t\in P^\perp(\Omega)$, gives the claim.
\end{proof}

\section{Proof of the endpoint Fefferman-Stein inequality}
We begin with a simple lemma.
\begin{lemma}\label{l.sub}
Let $\epsilon>0$, $f$ be a locally integrable function and set $F\coloneqq \{x\in\R^n: M_nf(x)>1\}$. There exists a finite collection of rectangles $\mathcal R= \{R_k ^s\}_{1\leq k \leq N}$ such that:
\begin{itemize}
 \item[(i)] The collection $\mathcal R$ has the property \eqref{e.p2} with parameter $\epsilon$.
\item[(ii)] For each $1\leq k \leq N$ we have
\begin{align*}
 |R_k ^s|<\int_{R_k ^s} |f(y)| dy
\end{align*}
\item [(iii)] We have the estimate
\begin{align*}
 w(F)\lesssim_{\epsilon,w,n} w(\cup_k R_k ^s).
\end{align*}
\end{itemize}
\end{lemma}
\begin{proof} For every $x\in F$ let $R_x\in \mathfrak R_n$ be a rectangle such that 
	\begin{align*}
		\frac{1}{|R_x|}\int_{R_x}|f(y)|dy>1.
	\end{align*}
Without loss of generality we may assume that $\{R_x\}_{x\in F}$ is a finite sequence $\{R_j\}_{1\leq j\leq M}$, so that (ii) is satisfied and such that $w(F)\leq w(\cup_{1\leq j\leq M} R_j)$. From the collection $\{R_j\}_{1\leq j \leq M}$ we will now choose a subcollection $\{R_k ^s\}_{1\leq k\leq N}$ so that (i) and (iii) are also satisfied. First we reorder the rectangles $R_j$ so that $P^\perp(R_1)\geq P^\perp(R_2)\geq \cdots\geq P^\perp(R_M)$. We choose $R_1 ^s\coloneqq R_1$ and assume that the rectangles $R_1 ^s, R_2 ^s,\ldots, R_\tau ^s,$ have been selected. Also let $1\leq j_o< M$ so that $R_\tau ^s=R_{j_o}$. We then choose $R_{\tau+1} ^s$ to be the rectangle with the smallest index among the rectangles $S\in\{R_{j_o+1},\ldots, R_M\}$ that satisfy
\begin{align*}
 |S\cap \bigcup_{j\leq \tau} (R_j ^s)^*|\leq \epsilon|S|.
\end{align*}
Since the collection $\{R_j\}_{1\leq j \leq M}$ is finite, the selection process will end after a finite number of $N$ steps, and the collection $\{R_k ^s\}_{1\leq k\leq N}$ will automatically satisfy (i). Of course this subcollection still satisfies (ii). Now assume that some $S\in\{R_1,\ldots,R_M\}$ was not selected. We can then find some positive integer $K\in\{1,2,\ldots,N\}$ such that
\begin{align*}
 |S\cap \bigcup_{j\leq  K} (R_j ^s)^*|> \epsilon|S|.
\end{align*}
Thus we get for all $x\in S$
\begin{align*}
M_n(\ind_{\cup_{j\leq  N} (R_j ^s)^*})(x)\geq M_n(\ind_{\cup_{j\leq  K} (R_j ^s)^*})(x)>\epsilon,
\end{align*}
which means that
\begin{align*}
 \bigcup_{\substack{1\leq j\leq N \\ R_j\, \text{not selected} }} R_j \subseteq \{x: M_n(\ind_{\cup_{j\leq N} (R_j ^s)^*})(x)>\epsilon \}.
\end{align*}
However, since $w\in A_\infty ^*$ we know that $M_n:L^{p_o}(w) \to L^{p_o,\infty}(w)$ for some $p_o>1$. We conclude that
\begin{align*}
 w\big(\bigcup_{\substack{1\leq j\leq N \\ R_j\,\text{not selected} }} R_j\big)\lesssim_{\epsilon,w,n} w(\cup_{j\leq N} (R_j ^s)^*) \lesssim w(\cup_{j\leq N} R_j ^s).
\end{align*}
Thus $w(F)\leq w(\cup_{1\leq j\leq M} R_j)\lesssim_{\epsilon,w,n} w(\cup_{1\leq k\leq N} R_k ^s)$ as we wanted.
\end{proof}
We are now ready to give the proof of our main result.
\begin{proof}[Proof of the Main Theorem] We assume that $n\geq 2$ since in one dimension $M_1$ is the usual Hardy-Littlewood maximal function and there is nothing (new) to prove. We henceforth write $M$ for $M_n$ since the dimension $n$ is fixed throughout the proof. Furthermore, it suffices to prove the theorem for $\lambda=1$. Let $F \coloneqq \{x\in\R^n: M(f)(x)>1\}$ and consider the collection $\mathcal R\coloneqq \{R_k ^s\}_{k=1} ^N \subset\mathfrak R_n$ given by Lemma \ref{l.sub}. By (i) of that Lemma the collection $\mathcal R$ has the sparseness property \eqref{e.p2}. We assume that $\epsilon>0$ was chosen small enough in Lemma \ref{l.sub}, and thus in \eqref{e.p2}, so that Lemma \ref{l.t*1} is valid. Observe that \eqref{e.p2} also implies that
\begin{align*}
 \Abs{R_k ^s\cap \bigcup_{j<k} R_j ^s}\leq \epsilon |R_k ^s|.
\end{align*}
 By choosing $\epsilon>0$ small enough we can also assume that $w(R_k ^s \cap \bigcup_{j<k} R_j ^s)\leq \frac{1}{2} w(R_k ^s)$, according to Remark \ref{r.esmall}. Setting $E_k\coloneqq R_k ^s\setminus \bigcup_{j<k} R_j ^s$ we will thus have
\begin{align} \label{e.recmass}
w(R_k ^s)\geq w(E_k)\geq \frac{1}{2}w(R_k ^s),\quad k=1,2,\ldots,N,
\end{align}
and the choice of $\epsilon>0$ depends only on the weight $w\in A_\infty ^*$. Denoting $\Omega \coloneqq \bigcup_{k=1} ^N R_k ^s$, we use (ii) and (iii) of Lemma \ref{l.sub} to estimate
\begin{align*}
 w(F)&\lesssim_{\epsilon,w,n} w(\Omega) \leq \sum_{k=1} ^N w(R_k ^s)\leq \sum_{k=1} ^N \frac{w(R_k ^s)}{|R_k ^s|} \int_{R_k ^s} |f(y)|dy
\\
&= \int_{\Omega} f(x) \sum_{k=1} ^N \frac{w(R_k ^s)}{|R_k ^s|}\ind_{R_k ^s} (x) dx.
\end{align*}
Define the linear operators
\begin{align*}
 Tf(x)=\sum_{k=1} ^N \frac{1}{|R_k ^s|} \int_{R_k ^s} f(y)dy \ind_{E_k }(x),\quad	
T^*f(x) = \sum_{k=1} ^N \frac{1}{|R_k ^s|} \int_{E_k} f(y)dy \ind_{R_k ^s}(x),\quad x\in\R^n.
\end{align*}
For locally integrable $f,g$ we have
\begin{align*}
 \int_{\Omega} Tf(x)g(x)dx = \int_{\Omega} T^*g(x) f(x) dx, \quad Tf(x)\leq M f(x),\quad x\in \R^n.
\end{align*}
By \eqref{e.recmass} we have
\begin{align*}
T^* w(x)=\sum_{k=1} ^N \frac{w(E_k ^s)}{|R_k ^s|} \ind_{R_k ^s}(x) \simeq \sum_{k=1} ^N \frac{w(R_k ^s)}{|R_k ^s|} \ind_{R_k ^s}(x)
\end{align*}
 thus we can estimate for any $\delta>0$
\begin{align*}
 w(\Omega)&\lesssim \int_\Omega f T^* w\leq \int_{\{\Omega:T^* w\leq \delta Mw\}} f(x) T^* w(x)dx + \int_{\{\Omega:T^* w>\delta Mw\}} f(x) \frac{T^* w(x)}{Mw(x)} Mw(x) dx
\\
&\leq \delta \int_{\R^n}  |f(x)|M w(x)dx+ \int_{\{\Omega:T^* w>\delta Mw\}} f(x) \frac{T^* w(x)}{Mw(x)} Mw(x) dx.
\end{align*}
We will use the following elementary estimate: For each $\theta>0$ there exists a constant $c_\theta>0$ such that for all $s,t\geq 0$ we have
\begin{align}
 st \leq c_\theta s[1+(\log^+s)^{n-1}]+\exp(\theta t^\frac{1}{n-1})-1,\quad n\geq 2.
\end{align}
The interested reader can find a detailed proof of this classical inequality in \cite{Bagby}. Applying this pointwise estimate we get for every $\theta>0$:
\begin{align*}
 w(\Omega)&\lesssim(\delta+c_\theta)  \int_{\R^n} |f(x)|\Big(1+\big(\log^+|f(x)|\big)^{n-1} \Big)Mw(x)dx
\\
&\quad \quad +\int_{\{\Omega:\ T^*w >\delta M w\}} \bigg( \exp \bigg[ \theta  \bigg( \frac{T^* w(x)}{Mw(x)} \bigg)^\frac{1}{n-1} \bigg] -1\bigg)Mw(x)dx.
\end{align*}

We now estimate the last term,

\begin{align*}
 Q&\coloneqq \int_{\{\Omega:\ T^*w >\delta M w\}} \bigg(\exp \bigg[ \theta  \bigg( \frac{T^* w(x)}{Mw(x)} \bigg)^\frac{1}{n-1} \bigg] -1\bigg) Mw(x)dx
\\
& =\sum_{k=1} ^\infty \frac{\theta^k}{k!} \int_{\{\Omega:\ T^*w>\delta M w\}} \bigg(\frac{T^*w(x)}{Mw(x)}\bigg )^\frac{k}{n-1}  Mw(x)  dx \leq \sum_{1\leq k\leq n-1} +\sum_{k>n-1}\eqqcolon I+II.
\end{align*}
For $I$ we just observe that since $k/ (n-1)\leq 1$ and $T^*w/(\delta Mw)> 1$ we have the elementary estimate
\begin{align*}
 \big(T^*w/Mw\big)^\frac{k}{n-1}& =\big(T^* w/(\delta Mw)\big)^\frac{k}{n-1} \delta^\frac{k}{n-1}
\\
&=\delta^{\frac{k}{n-1}-1} \big(T^* w/(\delta Mw)\big)^{\frac{k}{n-1}-1} \frac{T^*w}{Mw}
\\
& \leq \delta^{\frac{k}{n-1}-1} \frac{T^*w}{Mw}.
\end{align*}
So we have
\begin{align*}
 I \leq \sum_{1\leq k\leq n-1} \frac{\theta^k \delta^{\frac{k}{n-1}-1} }{k!} \int_\Omega T^* w(x)dx \leq \frac{\theta}{\delta} e^{\delta^\frac{1}{n-1}} \int_\Omega T1(x)w(x)\lesssim_{\delta,n}  \theta  w(\Omega).
\end{align*}
Here we abuse notation by denoting $T1, T^*1(x)$ the action of $T,T^*$, respectively, on the constant function $1$. For $II$ we use the fact that $T^*w\simeq \sum_{k=1} ^N \frac{w(R_k ^s)}{|R_k ^s|}\ind_{R_k ^s} \leq Mw\sum_{k=1} ^N \ind_{R_k ^s}\simeq Mw T^* 1$. We have
\begin{align*}
 II&\leq \sum_{k>n-1}\int_\Omega \frac{\theta^k}{k!} \bigg(\frac{T^*w(x)}{Mw(x)}\bigg )^{\frac{k}{n-1}-1} \frac{T^*w(x)}{Mw(x)}  Mw(x)  dx
\\
&\lesssim\sum_{k>n-1}\frac{\theta^k}{k!} \int_\Omega (T^*1(x))^{\frac{k}{n-1}-1} T^*w (x)dx
\\
&\lesssim \sum_{k>n-1}\frac{\theta^k}{k!} \int_\Omega (T^*1(x))^{\frac{k}{n-1}} T^*w (x)dx\quad \bigg(\text{because}\ T^*1\gtrsim 1\ \text{on}\ \Omega\bigg)
\\
&\lesssim \sum_{k>n-1}\frac{\theta^k}{k!} \int_\Omega T(  T^*(1) ^\frac{k}{n-1} )(x)w(x)dx\eqqcolon \sum_{k>n-1} \frac{\theta^k}{k!}Q_k.
\end{align*}
Since $w\in A_{p_o} ^*$ for some $1< p_o <\infty$ and $Tf\leq Mf$ we have $\|T(f)\|_{L^{p_o}(w)}\lesssim_{w,n} \|f\|_{L^{p_o}(w)}$. This together with Lemma \ref{l.t*1} yields
\begin{align*}
 Q_k \lesssim_{w,n} w(\Omega)^\frac{1}{p_o '}\bigg( \int_\Omega \abs{T^*1(x)}^\frac{kp_o}{n-1} w(x)dx \bigg)^\frac{1}{p_o}\lesssim_{w,n} [kp_o/(n-1)]^{k} w(\Omega).
\end{align*}
Overall we get
\begin{align*}
 II\lesssim_{w,n} &\sum_{k>n-1} \frac{\theta^k}{k!} \frac{(kp_o)^k}{(n-1)^k} w(\Omega) \lesssim\sum_{k>n-1} \frac{({\theta e p_o}/{(n-1)})^k}{\sqrt{k}}  w(\Omega)
\\
&\lesssim \frac{(\theta e p_o/(n-1))^{n}}{\sqrt{n}}w(\Omega),
\end{align*}
if $\theta$ is small enough. Thus $Q\lesssim_{w,n} \theta w(\Omega)$. We have proved that for $\theta>0$ small and fixing $\delta=1$ (say) in the previous estimates we have
\begin{align*}
	w(\Omega)\lesssim_{w,n}\theta w(\Omega)+ (1+c_\theta)\int_{\R^n} |f(x)|\Big(1+\big(\log^+|f(x)|\big)^{n-1} \Big)Mw(x)dx.
\end{align*}
Choosing $\theta>0$ sufficiently small we thus have
\begin{align*}
	w(F)\lesssim w(\Omega)\lesssim_{w,n} \int_{\R^n} |f(x)|\Big(1+\big(\log^+|f(x)|\big)^{n-1} \Big)Mw(x)dx,
\end{align*}
which is the desired estimate.
\end{proof}

We have actually proved the following weighted analogue of the C\'ordoba-Fefferman covering lemma from \cite{CF}.
\begin{lemma} Let $w\in A_\infty ^*$. Suppose that $\{R_j\}_{j\in J}$ is a finite sequence of rectangles from $\mathfrak R_n$. Then there exists a subcollection $\{R_k ^s\}_{1\leq k \leq N}\subset \cup_{j\in J}R_j$ such that
\begin{list}{}{}
	\item[(i)] $\quad w(\cup_{j\in J}R_j) \lesssim_{w,n} w(\cup_{k=1} ^N R_k ^s).$
	\item[(ii)] For every $\delta>0$ there exists $\theta_o=\theta_o(\delta,w,n)>0$ such that, for every $\theta<\theta_o$ we have
	\begin{align*} \int_{\{\Omega:\ T^*w(x)>\delta M w(x)\}} \bigg(\exp \bigg[ \theta  \bigg( \frac{T^* w(x)}{Mw(x)} \bigg)^\frac{1}{n-1} \bigg] -1\bigg) Mw(x)dx\lesssim_{w,n,\theta,\delta} w(\cup_{k=1} ^N R_k ^s) .\end{align*}
\end{list}
Here $T^*w=\sum_{k=1} ^N \frac{w(R_k ^s)}{|R_k ^s|}\ind_{R_k ^s}$ and $M$ denotes the strong maximal function.
\end{lemma}


\end{document}